\documentclass[11pt,a4paper]{amsart}

\usepackage[english]{babel}
\usepackage[T1]{fontenc}
\usepackage[utf8]{inputenc}
\usepackage{amsmath}
\usepackage{amssymb}
\usepackage{amsthm}
\usepackage{latexsym}
\usepackage{amsfonts}

\usepackage[onehalfspacing]{setspace} 
\setdisplayskipstretch{.421} 
\newlength{\abovebis} 
\newlength{\belowbis} 
\newlength{\aboveshortbis} 
\newlength{\belowshortbis} 
\everydisplay\expandafter{%
  \the\everydisplay 
  \advance\abovedisplayskip\abovebis 
  \advance\belowdisplayskip\belowbis 
  \advance\abovedisplayshortskip\aboveshortbis 
  \advance\belowdisplayshortskip\belowshortbis 
} 
\def\R{\mathbb{R}}

\def\C{\mathbb{C}}
\def\H{\mathbb{H}}
\def\B{\mathbb{B}}

\def\Div{\mathrm{div}}
\def\Sc{\textrm{Sc}}
\def\curl{\textrm{curl}}

\theoremstyle{plain}

\newtheorem{lemma}{Lemma}[section]

\newtheorem{proposition}[lemma]{Proposition}

\theoremstyle{definition}

\newtheorem{remark}{Remark}

\numberwithin{equation}{section}

\author{Matteo Santacesaria}
\address{Department of Mathematics, University of  Genoa, 16146, Italy.}
\email{matteo.santacesaria@unige.fi}

\begin{document}
\title[Calderon's inverse problem in 3D]{Note on Calder\'on's inverse problem \\for measurable conductivities}

\begin{abstract}
The unique determination of a measurable conductivity from the Dirichlet-to-Neumann map of the equation $\Div (\sigma \nabla u) = 0$ is the subject of this note. A new strategy, based on Clifford algebras and a higher dimensional analogue of the Beltrami equation, is here proposed. This represents a possible first step for a proof of uniqueness for the Calder\'on problem in three and higher dimensions in the $L^\infty$ case.
\end{abstract}

\maketitle

\section{Introduction}
Let $\sigma \in L^{\infty} (\Omega)$ be an isotropic electrical conductivity, where $\Omega \subset \R^n$, $n \geq 3$ is a bounded domain with connected complement  and $\sigma (x) \geq \sigma_0 > 0$ a.e. in $\Omega$. For every $f \in H^{1/2}(\partial \Omega)$ there exists a unique solution $u \in H^1(\Omega)$ of the Dirichlet problem for the conductivity equation
\begin{equation} \label{diri}
\Div (\sigma \nabla u) =0 \quad \text{in } \Omega, \quad u|_{\partial \Omega} = f.
\end{equation}
Then it is possible to define the Dirichlet-to-Neumann map $\Lambda_{\sigma} : H^{1/2}(\partial \Omega) \to H^{-1/2}(\partial \Omega)$ as follows:
\begin{equation}\notag
\Lambda_{\sigma}f = \left. \sigma\frac{\partial u}{\partial \nu}\right|_{\partial \Omega},
\end{equation}
where $\nu$ is the unit outer normal vector to $\partial \Omega$, $f \in H^{1/2}(\partial \Omega)$ and $u$ the unique $H^1(\Omega)$ solution of the Dirichlet problem \eqref{diri}. The derivative $\sigma \partial u/ \partial \nu$ is defined by
\begin{equation}\notag
\langle \sigma\frac{\partial u}{\partial \nu}, \psi \rangle_{H^{-1/2}(\partial \Omega),H^{1/2}(\partial \Omega)} = \int_{\Omega} \sigma \nabla u \cdot \nabla \psi \,dx,
\end{equation}
where $\psi \in H^1(\Omega)$ and $dx$ is the Lebesgue measure.

In 1980, Calder\'on proposed the following inverse problem \cite{Calderon1980}.

{\bf Calderon's problem.} Given $\Lambda_{\sigma}$, find $\sigma$ in $\Omega$.

This inverse problem have triggered an impressive amount of pure and applied research in the last decades. Global uniqueness, meaning the injectivity of the map $\sigma \mapsto \Lambda_\sigma$, under some smoothness assumptions on $\sigma$, has been first shown in \cite{Sylvester1987} in three and higher dimensions, and in \cite{Nachman1996} in two dimensions. The latter result have been greatly improved in \cite{Astala2006a}, where uniqueness was obtained for measurable conductivities in two dimensions (later generalized in \cite{Astala2012}). In higher dimensions this problem is still open. The best results so far show that the lowest regularity required to guarantee uniqueness is Sobolev $W^{1,n}$ in dimension $n = 3,4$ \cite{Haberman2015} and Lipschitz in higher dimensions \cite{Caro2016}. 

It is unclear if global uniqueness in three and higher dimensions for measurable conductivites holds true. No counterexamples have been found but it has been conjectured \cite{Brown2003} that the lowest regularity possible is $W^{1,n}$, in dimension $n \geq 3$ (because of related results on unique continuation). It is rather clear, though, that the techniques used until now have reached some sort of limit and a new framework must be introduced in order to tackle the problem.

The present note suggests a new strategy to study this problem. The main idea is to extend the two-dimensional approach of Astala-P\"aiv\"arinta \cite{Astala2006a} to higher dimensions. It seemed that the most natural framework to do so is via Clifford algebras, which in the three dimensional case is the algebra of quaternions. The main result obtained in this note is to rewrite the conductivity equation, in the three dimensional case, as a higher dimensional analogue of the Beltrami equation, also known as Clifford-Betrami equation:
\begin{equation*} 
D F  = \mu D \bar F,
\end{equation*}
where $F$ is a Clifford algebra valued function and $D$ is a so-called Cauchy-Riemann operator. Incidentally, the Beltrami coefficient $\mu$ coincides with the one from \cite{Astala2006a}.

The next natural step is to construct so-called complex geometrical optics (CGO) solutions (also known as exponentially growing or Faddeev-type solutions \cite{Faddeev1966}) for this equation and study their properties. In this way one could obtain either a higher dimensional analogue of the $\bar \partial$ equation in some parameter space, a linear (or nonlinear) transform of $\mu$ from high frequency asymptotics, or other indirect information on the unknown conductivity. Here we only propose a possible definition of CGO solutions, leaving their construction and analysis to future work.

The proposed CGO solutions are characterized by an asymptotic behaviour defined by a new family of exponential functions inspired by \cite{li1994}. These exponential functions are not only harmonic, but also monogenic, i.e. they belong to the kernel of $D$. To show the usefulness of these functions, a new proof of uniqueness for the linearized Calder\'on problem at a constant conductivity is given.

Note that quaternionic analytic techniques have been used in connection with the inverse conductivity problem also in the works \cite{belishev2005,belishev2017,belishev2017a,delgado2017,delgado2018}.

The structure of this note is the following. In Section~\ref{sec:prel} we present the main ideas. Some basic notations of Clifford analysis are introduced, as well as the reduction of the conductivity equation to a Clifford-Beltrami equation and the new uniqueness proof for the linearized problem. We then propose a possible definition of CGO solutions in Section~\ref{sec:CGO}.

\section{The Clifford-Beltrami equation}\label{sec:prel}

Following the same argument as in \cite[Section 2]{Astala2006a}, it is possible to reduce the problem to the case where $\Omega$ is a smooth domain, for instance the unit ball in $\R^n$. Let us briefly review it.

The map $\Lambda_{\sigma}$ can be defined on general domains by identifying $H^{1/2}(\partial \Omega) = H^1(\Omega)/H^1_0(\Omega)$ and $H^{-1/2}(\partial \Omega) = H^{1/2}(\partial \Omega)^*$. The Dirichlet condition in \eqref{diri} is defined in the Sobolev sense, requiring $u-f \in H^1_0(\Omega)$ for $f \in H^{1/2}(\partial \Omega)$.

Now let $\B \subset \R^n$ be the unit ball, $\Omega \subset \B$ a simply connected domain and $\sigma_1, \sigma_2$ two $L^{\infty}$ conductivities defined in $\Omega$ such that $\Lambda_{\sigma_1}  = \Lambda_{\sigma_2}$. Extend $\sigma_1$ and $\sigma_2$ as the constant 1 outside $\Omega$ and let $\tilde \sigma_1, \tilde \sigma_2$ be the new conductivities on $\B$. For $f \in H^{1/2}(\partial \B)$ let $\tilde u_1 \in H^1(\B)$ the solution of $\Div (\tilde \sigma_1 \nabla \tilde u_1) = 0$ in $\B$ and $\tilde u_1|_{\partial \B} = f$. Now let $u_2 \in H^1(\Omega)$ be the solution to
\[
\Div (\sigma_2 \nabla u_2) = 0 \quad \text{in } \Omega, \qquad \tilde u_1 - u_2 \in H^1_0(\Omega),
\] 
and define $\tilde u_2 = u_2 \chi_\Omega + \tilde u_1 \chi_{\B \setminus \Omega} \in H^1(\B)$, because zero extensions of functions in $H^1_0(\Omega)$ belong to the $H^1$ class. Since $\Lambda_{\sigma_1} = \Lambda_{\sigma_2}$, we have that $\tilde u_2$ satisfies
\[
\Div (\tilde \sigma_2 \nabla \tilde u_2) = 0 \quad \text{in } \B.
\]
Note that in $\B \setminus \Omega$ we have $\tilde u_1 = \tilde u_2$ and $\tilde \sigma_1 = \tilde \sigma_2$. This immediately yields $\Lambda_{\tilde \sigma_1} f = \Lambda_{\tilde \sigma_2} f$, for every $f \in H^{1/2}(\partial \B)$. Thus, if uniqueness hold in $\B$, one obtain $\tilde \sigma_1 = \tilde \sigma_2$, and so $\sigma_1 = \sigma_2$.\smallskip

From now on we assume that $\Omega = \B$ and extend $\sigma \equiv 1$ outside $\Omega$.

We will rewrite the conductivity equation using some notation from differential geometry. Let $d$ be the exterior derivative on differential forms and $\star$ the Hodge star operator. Then, the conductivity equation \eqref{diri} can be written as

\begin{equation} \label{condif}
d \star (\sigma d u) = 0.
\end{equation}

We will now extend the notion of $\sigma$-harmonic conjugate, as considered in \cite{Astala2006a} on the plane, to higher dimensions. Similar ideas have been already explored in earlier works \cite{Alessandrini1994,Bers1954,Vekua1962}, 

\begin{lemma}
Let $\Omega \subset \R^n$, $n \geq 3$ be the unit ball, $\sigma \in L^{\infty}(\Omega)$ bounded from below and $u \in H^1(\Omega)$ a solution of the conductivity equation \eqref{condif}. Then there exists a $n-2$ form $\omega$, unique up to $d \phi$, for a $n-3$ form $\phi$, such that
\begin{align} \label{rel}
&d \omega = \star \sigma d u,\\
&d \star \left(\frac{1}{\sigma} d \omega\right) = 0.
\end{align}
\end{lemma}

\begin{proof}
The proof follows from Poincar\'e lemma and the properties of the Hodge star operator.
\end{proof}

The form $\omega$ will be called $\sigma$-harmonic conjugate of $u$ and it is (locally) given by $\frac{n(n-1)}{2}$ functions.

We will now restrict ourselves to the case $n=3$. In this case there exists three functions $u_1,u_2,u_3$ such that identity \eqref{rel} can be written as

\begin{equation} \label{rel2}
\sigma \nabla u = \curl (u_1,u_2,u_3).
\end{equation}
The triplet $(u_1,u_2,u_3)$ is defined up to $\nabla \phi$ for some function $\phi$, which will be precised later.

The system \eqref{rel2} is a 3D analogue of the one obtained in \cite{Astala2006a} on the plane. In that case this was equivalent to a Beltrami equation and thanks to the Ahlfors-Vekua theory of quasiconformal maps it was possible to construct CGO solutions for $L^{\infty}$ conductivities.

It does not seem clear how to construct CGO solutions directly for \eqref{rel2}. We will instead use the framework of Clifford analysis and Dirac operators~\cite{Brackx1982} to write the system in a more convenient form.\smallskip

We consider $\R_{(2)}$, the real universal Clifford algebra over $\R^2$. It is generated as an algebra over $\R$ by the elements $\{e_0,e_1,e_2\}$, where $e_1, e_2$ is a basis of $\R^2$ with $e_i e_j + e_j e_i = -2 \delta_{ij}$, for $i,j=1,2$, and $e_0 = 1$ is the identity and commutes with the basis elements. The dimension of $\R_{(2)}$ is $4$ and it can be identified with $\H$, the algebra of quaternions. We denote $e_3 = e_1 e_2$ for the sake of simplicity. An element of $\R_{(2)}$ can be written as
\begin{equation}
A = A_0 e_0 + A_1 e_1+A_2 e_2+A_3 e_3,
\end{equation}
where $A_j$, $j =0,\ldots,3$ are real. We define the conjugate $\bar A$ of an element $A$ as
\begin{equation}\label{def:conj}
\bar A = A_0 e_0 - A_1 e_1-A_2 e_2-A_3 e_3.
\end{equation}
For $A,B \in \R_{(2)}$ we write $AB$ for the resulting Clifford product. The product $\bar A B$ defines a Clifford valued inner product on $\R_{(2)}$. We have $\overline{AB} = \bar B \bar A$ and $\bar{\bar A} = A$. For $A \in \R_{(2)}$, $\Sc(A)$ denotes the scalar part of $A$, that is the coefficient of the element $e_0$. The scalar part of a Clifford inner product, $\Sc(\bar A B)$, is the usual inner product in $\R^4$ when $A$ and $B$ are identified as vectors. We will write it $\langle A, B \rangle$.

With this inner product the space $\R_{(2)}$ is an Hilbert space and the resulting norm is the usual Euclidean norm $\|A\|=(\sum_j A_j^2)^{1/2}$. A Clifford valued function $f : \R^3 \to \R_{(2)}$ can be written as $f = f_0 e_0 + f_1 e_1 + f_2 e_2+f_3 e_3$, where $f_j$ are real valued.

The Banach spaces $C^\alpha$, $L^p$, $W^{1,p}$ of $\R_{(2)}$-valued functions are defined by requiring that each component $f_j$ belong to such spaces. On $L^2(\Omega)$ we introduce the $\R_{(2)}$-valued inner product
\begin{equation}\notag
(f,g) = \int_{\Omega} \bar f(x) g(x) dx.
\end{equation}
We define the following Cauchy-Riemann operators, with $(x_0,x_1,x_2)$ coordinates of $\R^3$,
\begin{equation}\nonumber
D = \frac{\partial}{\partial x_0}+ e_1 \frac{\partial}{\partial x_1}+ e_2 \frac{\partial}{\partial x_2}
\end{equation}
and
\begin{equation}\nonumber
\bar D = \frac{\partial}{\partial x_0}- e_1 \frac{\partial}{\partial x_1}-  e_2 \frac{\partial}{\partial x_2}.
\end{equation}
The operator $\partial =  e_1 \frac{\partial}{\partial x_1}+ e_2 \frac{\partial}{\partial x_2}$ is called the Dirac operator. On a Clifford valued function $f = \sum_{k=0}^3 f_k e_k$, the operators $D$ and $\bar D$ can act from left and right:
\begin{equation}\notag
D^{(l)}f = \sum_{j=0}^2 \sum_{k=0}^3  \frac{\partial f_k}{\partial x_0}e_j e_k, \qquad D^{(r)}f = \sum_{j=0}^2 \sum_{k=0}^3  \frac{\partial f_k}{\partial x_0}e_k e_j,
\end{equation}
and the same for $\bar D^{(l)}, \bar D^{(r)}f$. A function $f$ is said to be \textit{left (right) monogenic} if $D^{(l)}f = 0$ ($D^{(r)}f = 0$). From now on we will denote $D^{(l)}f$ simply by $Df$.
We have that $D \bar D = \bar D D = \Delta$ where $\Delta$ is the Dirac Laplacian. 

Using these operators we can write the system \eqref{rel2} in a compact form. The Clifford valued function $F$ defined as
\begin{equation}\notag
F = u e_0 + u_2 e_1 - u_1 e_2 - u_0 e_3,
\end{equation}
satisfies the following Clifford--Beltrami equation
\begin{equation} \label{bel}
D F  = \mu  D \bar F,
\end{equation}
where $\mu = (1-\sigma)/(1+\sigma)$, provided
\begin{equation*}
\Div (u_0,u_1,u_2) =0.
\end{equation*}
This last condition can always be achieved since we can add to $(u_0,u_1,u_2)$ the gradient of a function $\phi$ such that $\Delta \phi = -\Div (u_0,u_1,u_2)$. In other words, equation \eqref{bel} is equivalent to the system
\begin{equation}\label{eq:sys}
\left\{ \begin{array}{l}
\curl (u_0,u_1,u_2)=\sigma \nabla u,\\
\Div (u_0,u_1,u_2) =0.
\end{array}\right.
\end{equation}
More precisely, the following identities hold:
\begin{align}\label{eq:id1}
\frac{D \bar F + D F}{2} &= \frac{\partial u}{\partial x_0}e_0+ \frac{\partial u}{\partial x_1}e_1 + \frac{\partial u}{\partial x_2}e_2,\\ \label{eq:id2}
\frac{D \bar F - D F}{2} &= \curl_0 e_0 + \curl_1 e_1 + \curl_2 e_2 + \Div (u_0,u_1,u_2) e_3,
\end{align}
where we have denoted $(\curl_0,\curl_1,\curl_2) = \curl(u_0,u_1,u_2)$.

A generalization of Alessandrini's identity can be now readily proven.

\begin{proposition}
Let $\sigma_1,\sigma_2 \in L^\infty(\Omega)$ be two conductivities with $\sigma_1(x), \sigma_2 (x) \geq \sigma_0 >0$ a.e. in $\Omega$, and $\Lambda_1, \Lambda_2$ the associated Dirichlet-to-Neumann map, respectively. Then, for every $f_1,f_2 \in H^{1/2}(\partial \Omega)$ we have the identity
\begin{equation}\notag
\langle f_1,(\Lambda_2-\Lambda_1)f_2 \rangle_{H^{1/2}(\partial \Omega),H^{-1/2}(\partial \Omega)}=\frac 1 2  \int_\Omega (\mu_1-\mu_2)\langle D\bar{F_1} , D \bar{F_2}\rangle dx,
\end{equation}
where $u_j = \mathrm{Sc}(F_j)$ solves
\begin{equation}\notag
\mathrm{div} (\sigma_j \nabla u_j) = 0, \; \text{in } \Omega, \quad u_j = f_j, \; \text{on } \partial \Omega,
\end{equation}
and $F_j$ satisfy $D F_j = \mu_j D \bar F_j$ in $\Omega$, with $\mu_j = (1-\sigma_j)/(1+\sigma_j)$, $j = 1,2$.
\end{proposition}

\begin{proof}
By Green's formulas one readily obtain the classical Alessandrini's identity for the Calder\'on problem \cite{Alessandrini1988}:
\begin{equation*}
\langle f_1,(\Lambda_2-\Lambda_1)f_2 \rangle_{H^{1/2},H^{-1/2}}= \int_\Omega (\sigma_2-\sigma_1)\nabla u_1 \cdot \nabla u_2\, dx.
\end{equation*}
Let now $U_j$ the vector field such that $\curl( U_j) = \sigma_j \nabla u_j$. Then
\begin{equation*}
\langle f_1,(\Lambda_2-\Lambda_1)f_2 \rangle_{H^{1/2},H^{-1/2}}= \int_\Omega \curl (U_2) \cdot \nabla u_1 - \curl(U_1)\cdot \nabla u_2\, dx.
\end{equation*}
Using identities \eqref{eq:id1}, \eqref{eq:id2}, and the definition of the scalar product $\langle \cdot , \cdot \rangle$ in $\R_{(2)}$, we can write the quantity under the integral sign as the scalar part of a Clifford product as follows:
\begin{align*}
&\frac 1 4 \left( \langle D \bar F_2 - D F_2, D \bar F_1 +D F_1 \rangle -\langle D \bar F_2 + D F_2, D \bar F_1 -D F_1 \rangle\right)\\
& \qquad = \frac 1 4 \left((1-\mu_2) ( 1+\mu_1) - (1+\mu_2) (1-\mu_1)\right) \langle D \bar F_2, D \bar F_1\rangle\\
& \qquad = \frac{\mu_1 - \mu_2}{2} \langle D \bar F_2, D \bar F_1\rangle
\end{align*}
thanks to the Clifford-Beltrami equation satisfied by $F_1, F_2$.
\end{proof}

We now consider the complex Clifford algebra $\C_{(2)}$, generated over $\C$ with the same basis elements of $\R_{(2)}$. Note that the Clifford conjugation is always defined as in \eqref{def:conj}, so that it does not extend to the complex conjugation on the coefficients (which is never used in this note). Following \cite{li1994}, we define the following exponential function with values in $\C_{(2)}$:
\begin{align}\label{def:E1}
E_1(x,\zeta) = \sum_{k=0}^\infty \frac{1}{k!}\left( i (x_1\zeta_1 + x_2 \zeta_2-x_0(\zeta_1 e_1 +\zeta_2 e_2))\right)^k,
\end{align}
for $x = (x_0,x_1,x_2) \in \R^3$, $\zeta = (\zeta_1,\zeta_2) \in \C^2$ and $i$ is the imaginary unit. It is a holomorphic function of $\zeta \in \C^2$ for each $x \in \R^3$ and satisfies
\begin{align*}
\frac{\partial}{\partial x_0}E_1(x,\zeta) = -i(\zeta_1 e_1+\zeta_2 e_2) E_1(x,\zeta) &= -\left(e_1 \frac{\partial}{\partial x_1}+e_2 \frac{\partial}{\partial x_2}\right) E_1(x,\zeta)\\
&= -\left( \frac{\partial}{\partial x_1}E_1 e_1+ \frac{\partial}{\partial x_2} E_1 e_2\right).
\end{align*}
This yields $D^{(l)} E_1 = D^{(r)} E_1= 0$, that is $E_1$ is left and right monogenic. Moreover we have that $E_1(x,\zeta)E_1(y,\zeta) = E_1(x+y,\zeta)$, $E_1(x,-\zeta) = E_1(-x,\zeta)$, and 
\begin{align*}
&E_1(x,\zeta) = e^{i (x_1\zeta_1 + x_2 \zeta_2-x_0(\zeta_1 e_1 +\zeta_2 e_2))} = e^{i (x_1\zeta_1 + x_2 \zeta_2)} e^{-ix_0(\zeta_1 e_1 +\zeta_2 e_2)}.
\end{align*}
Note that we also have
\begin{align*}
&E_1(x,\zeta) = e^{i (x_1\zeta_1 + x_2 \zeta_2)}\left(\cosh(x_0 |\zeta|_{\C}) - \frac{i(\zeta_1 e_1 +\zeta_2 e_2)}{|\zeta|_{\C}}\sinh(x_0 |\zeta|_{\C}) \right).
\end{align*}
Here we have denoted $|\zeta|_{\C}$ a square root of $|\zeta|_{\C}^2$, the holomorphic extension of the Euclidean norm $\|\xi\|^2$, for $\xi \in \R^2$, defined as
\begin{equation}\notag
|\zeta|_{\C}^2 = \zeta_1^2 + \zeta_2^2 = \|\xi\|^2 - \|\eta\|^2 +2 i \xi \cdot \eta,
\end{equation}
for $\zeta = \xi + i \eta \in \C^2$ (where $\xi, \eta \in \R^2$). See \cite[\S 2]{li1994} for more details.

We also introduce $E_2(x,\zeta) = \frac{1}{2|\zeta|_{\C}}\bar D E_1(x,\zeta) = -\frac{i(\zeta_1 e_1 +\zeta_2 e_2)}{|\zeta|_{\C}}E_1(x,\zeta)$, for $|\zeta|_{\C}\neq 0$, which is left and right monogenic and can be written as
\begin{align*} \label{def:E2}
&E_2(x,\zeta) = e^{i (x_1\zeta_1 + x_2 \zeta_2)}\left(\sinh(x_0 |\zeta|_{\C}) - \frac{i(\zeta_1 e_1 +\zeta_2 e_2)}{|\zeta|_{\C}}\cosh(x_0 |\zeta|_{\C}) \right).
\end{align*}
Now consider the following combination of the two:
\begin{equation}\notag
E(x,\zeta) = E_1(x,\zeta)-E_2(x,\zeta) = \left(1+\frac{i(\zeta_1 e_1 +\zeta_2 e_2)}{|\zeta|_{\C}}\right)E_1(x,\zeta).
\end{equation}
Using the identity $e^z = \cos(z)+i\sin(z)$, for $z \in \C$, one readily obtains:
\begin{equation}
E(x,\zeta) = e^{i(x_1\zeta_1 + x_2 \zeta_2) -x_0|\zeta|_{\C}}\left( 1+ \frac{i(\zeta_1 e_1 +\zeta_2 e_2)}{|\zeta|_{\C}}\right).
\end{equation}
This function is left and right monogenic and its scalar part coincides with the harmonic exponential of the classical CGO solutions, i.e. $e^{i x \cdot \zeta}$, for $\zeta \in \C^3$, $\zeta \cdot \zeta =0$.

Using the function $E$, it is possible to give a new proof of the uniqueness of the linearized Calder\'on problem at a constant conductivity. More precisely, we show the injectivity of the Fr\'echet derivative of the Dirichlet-to-Neumann map $d\Lambda |_{\sigma \equiv const.}$, a result originally obtained by Calder\'on \cite{Calderon1980}.

\begin{proposition}\label{prop:lin}
The Fr\'echet derivative $d \Lambda|_{\sigma \equiv const}$ at a constant conductivity is injective.
\end{proposition}
\begin{proof}
Following the same argument as in \cite{Calderon1980}, the statement is equivalent to show that, given $\delta \in L^{\infty}(\Omega)$, if
\begin{equation}\label{eq:lindelta}
\int_\Omega \delta \nabla u_1 \cdot \nabla u_2 dx = 0,	
\end{equation}
for every $u_1,u_2$ such that $\Delta u_1 = \Delta u_2 =0$, then $\delta \equiv 0$. Let now $u_1, u_2$ be the scalar part of $E$. Note that
\begin{align*}
&\Sc(E(x,\zeta)) = e^{i(x_1\zeta_1 + x_2 \zeta_2) -x_0|\zeta|_{\C}}.
\end{align*}
Let $k = (k_1,k_2,k_3) \in \R^3$ and $a,b \in \C^2$ be such that $a+b = (k_1,k_2)$ and $|a|_{\C}+|b|_{\C} = i k_3$. These parameters can be constructed for instance as $a = (k_1,k_2) - b$, $b = (\lambda, i \lambda)$ where $\lambda = \|k\|^2 / ( 2 (k_1 + i k_2))$, for $(k_1,k_2) \neq (0,0)$ and $a = -b = (0,ik_3/2)$ for $(k_1,k_2)=(0,0)$. Note the the choice of the square root of $|a|_\C^2$ is determined by the condition $|a|_{\C} = i k_3$.

Then, plugging $\Sc(E(x,a)), \Sc(E(x,b))$, into \eqref{eq:lindelta}, we find
\begin{align*}
(|a|_{\C}|b|_{\C}-a\cdot b ) \int_{\Omega} \delta(x) e^{i x \cdot k} dx=-\frac{\|k\|^2}{2} \int_{\Omega} \delta(x) e^{i x \cdot k} dx = 0,
\end{align*}
for every $k \in \R^3$. Thus the Fourier transform of $\delta$ vanishes and so $\delta \equiv 0$. 
\end{proof}

\section{CGO solutions of the Clifford--Beltrami equation}\label{sec:CGO}

The main ingredient of essentially every global uniqueness proof for Cal\-der\'on's problem is a special family of solutions of a certain equation, often referred as complex geometrical optics (CGO) solutions, with prescribed asymptotic behavior. 

The purpose of this section is to propose a definition of CGO solutions for $L^\infty$ conductivities. Their existence and properties are not studied in this note, since new Clifford analytic tools seems to be required. This could represent the main hurdle in the understanding of Calder\'on's problem for discontinuous conductivities in three or higher dimension. \smallskip

In view of the previous section, it seems natural to consider CGO solutions for the Clifford-Beltrami equation \eqref{bel}. From the uniqueness for the linearized problem (Proposition \ref{prop:lin}) the asymptotic behavior of these solutions should be dictated by the exponential function $E$ introduced in Section~\ref{sec:prel}.

We first need to establish a Leibniz formula for the operator $D$. This was already obtained in \cite[Theorem 1.3.2]{Gurlebeck1990} for a slightly different Cauchy-Riemann operator.
\begin{lemma}[Leibniz's formula]\label{lem:leib}
Let $f = \sum_{k=0}^3 f_k e_k, g = \sum_{l=0}^3 g_l e_l$ be two Clifford valued functions. Then
\begin{equation}\label{eq:leib}
D(fg) = (Df)g - \bar f (\bar D g) +2 \Sc(fD)g,
\end{equation}
where $\Sc(fD) = f_0\frac{\partial}{\partial x_0} -\sum_{k=1}^2 f_k\frac{\partial}{\partial x_k}$.
\end{lemma}
\begin{proof}
Let $\partial_j = \frac{\partial}{\partial x_j}$, $D = \sum_{j=0}^2 e_j \partial_j = \partial_0 + \partial$ and $\bar D = \partial_0 - \partial$. Denote $f = f_0 + f_I$ and $\bar f = f_0 - f_I$.  We have
\begin{align*}
D(fg) &= \sum_{j,k,l} \partial_j(f_k g_l) e_j e_k e_l = \sum_{j,k,l} (\partial_j f_k g_l + f_k \partial_j  g_l)  e_j e_k e_l\\
&= (Df)g + \sum_{j,k,l} f_k \partial_j  e_j e_k g_l e_l = (Df)g + \left(\sum_{j,k} f_k \partial_j  e_j e_k \right)g,
\end{align*}
where the sums are taken over all indices $j=0,1,2$ and $k,l = 0,1,2,3$. Now the last term can be rewritten as
\begin{align*}
&\left( \sum_{k=0}^3 f_k \partial_0 e_k \right)g + \left( \sum_{j=1}^2 f_0 \partial_j e_j \right)g +  \left( \sum_{j=1}^2 \sum_{k=1}^3 f_k \partial_j e_j e_k \right)g\\
&\quad = f \partial_0 g +f_0 \partial g - \left( \sum_{k=1}^3 f_k  e_k \sum_{j=1}^2 \partial_j  e_j \right)g - 2(f_1 \partial_1 +f_2 \partial_2)g\\
&\quad= -\bar f \partial_0 g + f_0 \partial g - f_I \partial g +2(f_0\partial_0-f_1 \partial_1 -f_2 \partial_2)g\\
&\quad = -\bar f (\bar D g) +2\, \Sc(f D)g.
\end{align*}
The proof follows by combining the two identities.
\end{proof}
\begin{remark}
If $f$ is a scalar function one recovers the classical Leibniz's formula
\[
D(fg) =D(gf)= (Df)g + f(Dg).
\]
\end{remark}\bigskip

We seek solutions to equation \eqref{bel} of the form
\begin{equation} \label{defF}
F(x,\zeta) = E(x,\zeta)M(x,\zeta),
\end{equation}
with
\begin{equation}\nonumber
M(x,\zeta) \to 1 \quad \text{as } |x| \to +\infty,
\end{equation}
for $\zeta \in \C^2$, $|\zeta|_{\C} \neq 0$, and $M$ a $\C_{(2)}$-valued function. Plugging \eqref{defF} into equation \eqref{bel}, using Leibniz's formula \eqref{eq:leib} and the fact that $D E= \bar D \bar E=0$, we find that $M = \sum_{j=0}^3 M_{j}e_j$ satisfies:
\begin{align}\notag
\bar D M(x,\zeta) = &-\frac{\mu}{2} \left( 1 + \frac{i\zeta}{|\zeta|_{\C}}\right) D \overline{ M (x,\zeta)}\left( 1 - \frac{i\zeta}{|\zeta|_{\C}}\right)\\ \notag
 &-2 \mu (-M_{0}|\zeta|_{\C}+i M_{1}\zeta_1 +iM_{2} \zeta_2 )\\ \label{eq:db}
&+ \left( 1 + \frac{i\zeta}{|\zeta|_{\C}}\right)\left(\frac{\partial}{\partial x_0} - i \frac{\zeta_1}{|\zeta|_{\C}}\frac{\partial}{\partial x_1} -i\frac{\zeta_2}{|\zeta|_{\C}}\frac{\partial}{\partial x_2} \right) M,
\end{align}
where we have denoted $\zeta = \zeta_1 e_1 + \zeta_2 e_2$.
We also used the fact that
\begin{align*}
&\overline{E(x,\zeta)} = e^{i(x_1\zeta_1 + x_2 \zeta_2) -x_0|\zeta|_{\C}}\left( 1- \frac{i \zeta}{|\zeta|_{\C}}\right),\\
&\qquad \left( 1- \frac{i \zeta}{|\zeta|_{\C}}\right)^{-1}= \frac 1 2 \left( 1+ \frac{i \zeta}{|\zeta|_{\C}}\right),
\end{align*}
since the Clifford conjugation does not change the complex coefficients.

Introducing the Clifford element $Z = 1 +\frac{i\zeta}{|\zeta|_{\C}} = 1 +\frac{i ( \zeta_1 e_1 + \zeta_2 e_2)}{|\zeta|_{\C}}$, we can rewrite equation \eqref{eq:db} as
\begin{align}
\bar D M(x,\zeta) = &- \mu Z\, D \bar M (x,\zeta)Z^{-1} + 2\mu |\zeta|_{\C} \Sc( Z M)+ Z \,\Sc(Z D) M.
\end{align}

\providecommand{\href}[2]{#2}
\providecommand{\arxiv}[1]{\href{http://arxiv.org/abs/#1}{arXiv:#1}}
\providecommand{\url}[1]{\texttt{#1}}
\providecommand{\urlprefix}{URL }

\end{document}